\documentclass[preprint,12pt]{elsarticle}
\usepackage{amsmath}
\usepackage{amsthm}
\usepackage{amssymb,epsfig}
\usepackage[cp1250]{inputenc}
\usepackage{a4}

\newtheorem{lemma}{Lemma}[section]
\newtheorem{prop}[lemma]{Proposition}

\newtheorem{thm}[lemma]{Theorem}
\theoremstyle{definition}
\newtheorem{definition}[lemma]{Definition}

\newtheorem{example}[lemma]{Example}
\newtheorem{claim}[lemma]{Claim}

\def\Z{\mathbb Z}
\def\N{\mathbb N}

\def\leq{\leqslant}

\def\geq{\geqslant}

\begin{document}
\begin{frontmatter}
\title{Palindromic closures using  multiple antimorphisms}

\author[FMFI UK]{Tatiana \textsc{Jajcayov\'a}}
\author[FJFI]{Edita \textsc{Pelantov\'a}}
\author[FIT]{\v St\v ep\'an \textsc{Starosta}\corref{cor1}}

\cortext[cor1]{Corresponding author}
\address[FMFI UK]{Department of Applied Informatics, FMFI Comenius University,\\
Mlynsk\' a dolina, 842 48 Bratislava, Slovakia}
\address[FJFI]{
Department of Mathematics, FNSPE, Czech Technical University in Prague, \\
Trojanova 13, 120 00 Praha 2, Czech Republic}

\address[FIT]{
Department of Applied Mathematics, FIT, Czech Technical University in Prague, \\
Th\' akurova 9, 160 00 Praha 6, Czech Republic}


\begin{abstract}
Generalized pseudostandard word  $\bf u$, as  introduced in 2006 by de Luca and De Luca, is given by a directive sequence of letters  from an alphabet ${\cal A}$ and by a directive sequence of involutory antimorphisms acting on ${\cal A}^*$.
Prefixes  of $\bf u$ with increasing length are constructed using pseudopalindromic closure operator.

\noindent We show that generalized Thue--Morse words ${\bf t}_{b,m}$, with $b, m \in \N$ and  $b, m  \geq 2$, are generalized pseudostandard words if and only if ${\bf t}_{b,m}$ is a periodic word or $b \leq m$. This extends the result of de Luca and De Luca obtained for the classical Thue--Morse words.
\end{abstract}


\begin{keyword}
palindromic closure, generalized Thue--Morse word, involutory antimorphism
\end{keyword}


\end{frontmatter}


\section{Introduction}

\textit{Palindromic closure} of a finite word $w$ is the shortest
palindrome having $w$ as a prefix. This concept was introduced in 1997
 by Aldo de Luca for words over the binary alphabet. In \cite{Lu}, De Luca
showed that any standard Sturmian word is a limit of a sequence of
 palindromes $(w_n)$, where $w_0$  equals the empty word and $w_{n+1}$
 is the  palindromic closure of $w_n\delta_{n+1}$ for some letter
$\delta_{n+1}$ from the binary alphabet. And, vice versa, the limit
of such a sequence is  always a standard Sturmian word.
 This construction was
extended to any finite alphabet $\mathcal{A}$ by Droubay, Justin
and Pirillo in \cite{DrJuPi}, and   words arising by their
construction are called \textit{standard episturmian words}. The sequence
$\delta_1\delta_2\delta_3\ldots $ of letters that are added successively
at each step is referred to as  the \textit{directive sequence} of the standard
episturmian word.

An important generalization of standard episturmian words appeared in
\cite{LuLu}, where the palindromic closure is replaced by
\textit{$\vartheta$-palindromic closure} with $\vartheta$  an arbitrary
involutory antimorphism of the free monoid $\mathcal{A}^*$. The
corresponding words are called \textit{$\vartheta$-standard words}, or
\textit{pseudostandard words}.

A further generalization was provided in 2008 by Michelangelo
Bucci, Aldo de Luca, Alessandro De Luca, Luca Q. Zamboni in
\cite{BuLuLuZa2}, where the sequence $(w_n)$
is allowed to start with an arbitrary finite word $w_0$, and the
limit word is a \textit{$\vartheta$-standard word with the seed $w_0$}.   
Words obtained by these generalizations are in some sense quite similar to  the standard episturmian words:
by results of Bucci and De Luca \cite{BuLu},
any $\vartheta$-standard word with a seed is a morphic image of
a standard episturmian word.

The described constructions of $(w_n)$ guarantee that the language of
a standard episturmian word contains  infinitely many palindromes.

Droubay, Justin and Pirillo in \cite{DrJuPi}  deduced that any
finite word $w$ contains at most $|w|+1 $ distinct palindromes  ($|w|$ stands for the length of $w$). 
A word $w$ with exactly $|w|+1$ palindromes is called \textit{rich} (in \cite{GlJuWiZa}), or  \textit{full}
(in \cite{BrHaNiRe}).  An infinite  word is \textit{rich} if every finite
 factor of this word is  rich.  Examples of rich words include all
episturmian words, see \cite{DrJuPi}, two interval coding of rotations,  see
\cite{BlBrLaVu11}, words coding interval exchange transformation
with symmetric permutation, see \cite{BaMaPe}, etc.

The notion of rich word was generalized as well: the concept of 
\textit{$\vartheta$-rich} word was introduced in \cite{Sta2010}.
A $\vartheta$-rich word is saturated by \textit{$\vartheta$-palindromes}, fixed points of the involutory antimorphism $\vartheta$, up to the highest possible level.
Another generalization of the concept of richness is a measure of how many palindromes are missing in a certain sense.
This quantity is called the \textit{palindromic defect} and it was first considered in \cite{BrHaNiRe}.
Analogously to standard episturmian words, one can prove that every $\vartheta$-standard word is
$\vartheta$-rich, and every standard word with seed has finite palindromic defect.

Again, both generalizations of rich words are not too far from the
original notion. In particular, any uniformly recurrent
$\vartheta$-rich word and any uniformly recurrent word with finite
defect is just a morphic image of a rich word, see \cite{PeSta_Milano_IJFCS}.
Nevertheless, the operation of palindromic closure enables us to
construct a big class of rich words.

Up to this point, we were concerned with properties of words with respect to one
fixed involutory antimorphism $\vartheta$ on $\mathcal{A}^*$.

In the last Section of the paper \cite{LuLu}, De Luca and de Luca introduced 
an even more general concept - generalized  pseudostandard words. They  considered a  set
$\mathcal{I}$ of involutory antimorphisms over $\mathcal{A}^*$ and
beside a directive sequence $\Delta =
\delta_1\delta_2\delta_3\ldots $   of letters from $\mathcal{A}$,
 also  a directive sequence of antimorphisms $\Theta =
\vartheta_1 \vartheta_2\vartheta_3\ldots $ from $\mathcal{I}$. The
construction of $(w_n)$ starts  with the empty word $w_0$ and
recursively, $w_{n}$ is the $\vartheta_n$-palindromic closure of
$w_{n-1}\delta_n$.

Then De Luca and de Luca focused on the prominent Thue--Morse word
${\bf u}_{TM}$.  They showed that ${\bf u}_{TM}$ is a generalized
pseudostandard word with the directive sequences $\Delta = 01111 \ldots =
01^\omega$ and $\Theta = RERERE \ldots = (RE)^\omega$, where $R$ denotes the mirror
image operator, and $E$ the antimorphism which exchanges the
letters $0\leftrightarrow 1$.

The example of the Thue--Morse word illustrates that generalized pseudostandard words substantially differ from the previous notions where only one antimorphism is used.
This is due to the fact that the Thue--Morse word is not a morphic  image of a standard episturmian word as all $\vartheta$-standard words are.
This can be seen when comparing the factor complexities of a standard episturmian word, which is of the form $an+b$ except for finitely many integers $n$ (see \cite{DrJuPi}), and the factor complexity of the Thue--Morse word (see \cite{Br89} or \cite{LuVa}).

The results of  \cite{BlPaTrVu} confirm that the  notion of
the generalized pseudostandard word is very fruitful. In particular,
Blondin-Mass\'e,  Paquin,  Tremblay, and Vuillon  showed that any
standard  Rote word is a generalized pseudostandard word. Since the
Rote words are defined over the binary alphabet, their directive sequences
$\Theta$ contain antimorphisms $R$ and $E$ only.

In this article we focus on the so-called generalized Thue--Morse words.
Given two integers $b$ and $m$ such that $b > 1$ and $m > 1$, we denote the generalized Thue--Morse word by ${\bf t}_{b,m}$.
The alphabet of ${\bf t}_{b,m}$ is $\mathcal{A}= \mathbb{Z}_m = \{0, \ldots, m-1 \}$.
 For a given integer base $b$, the number $s_b(n)$ denotes the digit sum
of the expansion of number $n$ in the base $b$. The word ${\bf
t}_{b,m}$ is defined
 \begin{equation}\label{DefThueMorse} {\bf
t}_{b,m}= \bigl(s_b(n) \bmod m\bigr)_{n=0}^\infty.
 \end{equation}
 In this
notation the classical Thue--Morse word equals ${\bf t}_{2,2}$. As
shown in \cite{Sta2011},  the language of ${\bf t}_{b,m}$ is closed
under a finite group containing $m$ involutory antimorphisms. This
group is isomorphic to the dihedral group $I_2(m)$.  Our aim in this paper is to prove  the following theorem:

\begin{thm}\label{main}  The generalized Thue--Morse word  ${\bf t}_{b,m}$ is a generalized pseudostandard word if and only
if \  $b\leq m$  \ or \  $ b-1 = 0 \pmod m$.
\end{thm}

Unlike the case of the  standard words with seed, very little is known
about the properties of generalized pseudostandard words. In the last
section, we  propose several questions the answering of which would bring
a better understanding of the structure of such words.

Our motivation for the study of generalized pseudostandard words
stems from a desire to find  $G$-rich words recently introduced in
\cite{PeSta1}. Words that are rich in the original sense are
 $G$-rich with respect to $G=\{R, Id\}$. In \cite{Sta2011},
the last author showed that the words ${\bf t}_{b,m}$ are $I_2(m)$-rich. In
particular, the classical Thue--Morse word  is $H$-rich with
$H=\{E,R, ER, Id\}$. Using the result of \cite{BlBrLaVu11}, one can also
show that  the Rote words are $H$-rich (for definition of the Rote
words see \cite{Rote}). These examples are almost all the examples
of $G$-rich words we know for which the group $G$  is not
isomorphic to $\{R, Id\}$. We believe that generalized
pseudostandard words can provide many other new examples.

\section{Preliminaries}

By $\mathcal{A}$ we denote a finite set of symbols  usually called the \textit{alphabet}.
A \textit{finite word} $w$ over  $\mathcal{A}$ is a string $w =w_0w_1\cdots w_{n-1}$ with $w_i \in \mathcal{A}$.
For its \textit{length} $n$ we write $|w|$.
The set of all finite words over $\mathcal{A}$, including the \textit{empty word} $\varepsilon$,  together with the operation of concatenation  of words, form the free monoid $\mathcal{A}^*$.
A \textit{morphism} of $\mathcal{A}^*$ is a mapping $\varphi: \mathcal{A}^* \to \mathcal{A}^*$ satisfying $\varphi(wv) = \varphi(w)\varphi(v)$ for all finite words $w,v \in \mathcal{A}^* $.
A morphism is uniquely given by the images $\varphi(a)$ of all letters $a \in \mathcal{A}$.
If, moreover, there exists a letter $b \in \mathcal{A}$ and a non-empty word $w \in \mathcal{A}^*$ such that $\varphi(b) = b w$, then the morphism $\varphi$ is called a \textit{substitution}.

If a word $w\in\mathcal{A}^*$  can be written as a concatenation $w= uvz$, then $v$ is called a \textit{factor} of $w$.
If $u$ is the empty word, then $v$ is called a \textit{prefix} of $w$; if $z$ is the empty word, then $v$ is called a \textit{suffix} of $w$.
If $\varphi$  is a substitution, then for every $n\in \mathbb{N}$ the word $\varphi^n(b)$ is a prefix of $\varphi^{n+1}(b)$.

An \textit{infinite word} ${\bf u}$ over  $\mathcal{A}$ is a sequence $u_0u_1u_2\cdots  \in  \mathcal{A}^\mathbb{N}$. 
The set of all factors  of ${\bf u}$ is denoted by $\mathcal{L}({\bf u})$ and
referred to as the \textit{language of ${\bf u}$}.
The action of a morphism $\varphi: \mathcal{A}^*\to \mathcal{A}^*$ can be naturally
extended to $\mathcal{A}^\mathbb{N}$  by $\varphi({\bf u}) =
\varphi(u_0u_1u_2\cdots )
=\varphi(u_0)\varphi(u_1)\varphi(u_2)\cdots $. If  $\varphi({\bf
u}) = {\bf u}$ for some infinite word ${\bf u}$, then the  word ${\bf
u}$ is a \textit{fixed point} of the morphism $\varphi$. Every
substitution $\varphi$ has a  fixed point, namely  the infinite
word which has prefix $\varphi^n(b)$ for every $n$; this infinite
word is denoted $ \varphi^\infty(b)$.

\begin{example}\label{exampleThueMorse} The Thue--Morse word  ${\bf
u}_{TM}$ is a fixed point of the substitution $\varphi_{TM}$ which
maps $0 \mapsto \varphi_{TM}(0)=01$  and $1 \mapsto
\varphi_{TM}(1)=10$. The substitution has two fixed points: the
Thue--Morse word ${\bf u}_{TM}= \varphi^\infty(0)$ and the word $\varphi^\infty(1)$.
\end{example}

When manipulating a fixed point ${\bf u}$ of a  substitution
$\varphi$, we will need the notion of an ancestor: We say that a word $w=
w_0w_1 \ldots w_k$ is  a {\it $\varphi$-ancestor}  of a word $v
\in \mathcal{L}({\bf u})$ if the following three conditions are satisfied:
 \begin{itemize}
   \item $v$ is a
factor of $\varphi(w_0w_1 \ldots w_k)$,
 \item  $v$ is not a factor
of $\varphi(w_1 \ldots w_k)$,
 \item $v$ is not a factor of
$\varphi(w_0w_1 \ldots w_{k-1})$.
 \end{itemize}
\begin{example}\label{exampleThueMorse}  Consider the Thue--Morse word  $${\bf
u}_{TM}=
 01101001100101101001011001\cdots =
\varphi_{TM}(0)\varphi_{TM}(1)\varphi_{TM}(1)\varphi_{TM}(0)\varphi_{TM}(1)\varphi_{TM}(0)\cdots
$$
The factor $v=010011$ has an ancestor $w = 1101$, since $v$ is a
factor of $\varphi_{TM}(1101) = 10100110$, and $w$ is neither a factor
of $\varphi_{TM}(101) = 100110$ nor a factor of $\varphi_{TM}(110) =
101001$. In fact, $w$ is the unique ancestor of $010011$.

The factor $v=010$ is a factor of $\varphi(11) =1010$ and a factor
of $\varphi(00)=0101$. Thus $v=010$ has two ancestors, namely $11$
and $00$.
\end{example}
Let us now define the key notion of this article. The mapping $\Psi:
\mathcal{A}^* \mapsto \mathcal{A}^*$ satisfying
$$\Psi (uv) = \Psi(v)  \Psi(u)\ \ \ \hbox{for any}  \ \ u,v \in
\mathcal{A}^*\ \quad  \hbox{and} \quad   \Psi^2  \ \ \hbox{equal to  the
identity}\,$$ is an \textit{involutory antimorphism}. Any antimorphism
$\Psi$ is determined by the images of the letters from
$\mathcal{A}$. 
The restriction of $\Psi$ to the alphabet $\mathcal{A}$
is a permutation  $\pi$ on $\mathcal{A}$ with  cycles of    length 1 or 2  only.

A word $w \in \mathcal{A}^*$ is a \textit{$\Psi$-palindrome} if
$\Psi(w)=w$.
The notion \textit{pseudopalindrome} is also used.

A \textit{$\Psi$-palindromic closure} of a factor $w \in \mathcal{A}^*$ is
the shortest $\Psi$-palindrome having $w$ as  a prefix. The
$\Psi$-palindromic closure of $w$ is denoted by $w^\Psi$.
If $w = uq$ such that $q$ is the longest $\Psi$-palindromic suffix of
$w$, then
$$ w^\Psi = uq\Psi(u).$$

\begin{example}\label{example1}  There are
two distinct involutory antimorphisms on the alphabet  $\mathcal{A} =\{0,1\}$:   $R$, the mirror image, and 
 $E$,  the antimorphism exchanging letters, i.e.,   $E(0)=1$
and $E(1)=0$. Put
 $ u= 0110110$. Then  $u$ is an $R$-palindrome, and
 thus  $u^R = u$. The word $u$ is not an $E$-palindrome, since
 $E(u)=E(0110110) = 1001001 \neq u$.  The $E$-palindromic closure
 of $u$ is $u^E = 011011001001$ since the longest $E$-palindromic
 suffix of $u$ is $10$.

\end{example}

In \cite{LuLu}, de Luca and De Luca generalized the notion of
standard words considering  the
set $\mathcal{I}$  of all involutory antimorphisms on
$\mathcal{A}^*$ instead of just one  fixed antimorphism. We will denote  by $\mathcal{I}^\mathbb{N}$ the set
of all infinite sequences over $\mathcal{I}$.

\begin{definition}\label{def:GPS}
 Let $\Theta  = \vartheta_1  \vartheta_2 \vartheta_3 \ldots
\in \mathcal{I}^\mathbb{N}$ and $\Delta =\delta_1\delta_2\delta_3
\ldots \in \mathcal{A}^\mathbb{N}$.   Denote
$$w_0=\varepsilon \quad \hbox{and} \quad
w_{n}= \Bigl(w_{n-1} \delta_{n}\Bigr)^{\vartheta_n} \quad
\hbox{for any }\ n \in \mathbb{N}, n\geq 1.$$  The word
$${\bf
u}_\Theta(\Delta) = \lim\limits_{n \to \infty}w_n$$   is called a
{\it generalized pseudostandard word} with the directive sequence of
letters $\Delta$ and the directive sequence of antimorphisms $\Theta$.
\end{definition}
Let us stress that the definition  of  ${\bf u}_\Theta(\Delta)$ is
correct as  $w_n$ is a prefix of $w_{n+1}$ for any $n$.

\begin{example}\label{example2}   Consider  the
directive sequence of letters  $\Delta = 0(101)^\omega$ and the
directive sequence of  antimorphisms $\Theta = (RE)^\omega$. Then

$w_0 = \varepsilon$

$w_1 = 0^R = 0$

$w_2 = (01)^E = 01$

$w_3 = (010)^R = 010$

$w_4 = (0101)^E = 0101$

$w_5 = (01011)^R = 010 11 010$

$w_6 = (010110100)^E = 010 110100 101$

$w_7= (0101101001011)^R = 010 1101001011 010$

$w_8= (0101101001011 0101)^E = 0101101001011 0101 0010110100101$

$\quad \vdots$

\end{example}

The authors of \cite{LuLu} proved that  the famous Thue--Morse word ${\bf u}_{TM}$ is a generalized pseudostandard word with directive sequences

 $$  \Delta = 01^\omega
 \quad \hbox{and} \quad \Theta = (ER)^\omega.$$

 \subsection{The Generalized Thue--Morse words and their properties}\label{T-M}

 In Introduction, we defined a generalized Thue--Morse word by \eqref{DefThueMorse}, i.e.,
its $n$-th letter is the digit sum of the expansion of $n$ in base $b$ taken modulo $m$.
It can be shown that ${\bf t}_{b,m}$ is a fixed point of the substitution  $\varphi_{b,m}$  over the alphabet $\Z_m$:

\begin{equation}\label{substitution}\varphi(k) = \varphi_{b,m}(k) =  k(k+1)(k+2)\ldots (k+b-1)\quad \hbox{for every}\  k  \in
\mathbb{Z}_m\end{equation} where letters are expressed modulo
$m$. As already stated in \cite{AlSh}, ${\bf t}_{b,m}$  is
periodic if and only if $b = 1 \pmod m$.

\noindent (Note about our subsequent notation: When dealing with letters from $\Z_m$, we will consider all operations modulo $m$. 
We will denote the relation $x = y \pmod m$ by $x =_m y$, to ease the notation.)

The language of ${\bf t}_{b,m}$  has many symmetries: 
denote  by  $I_2(m)$ the  group generated by antimorphisms
$\Psi_x$ defined for every  $x \in \mathbb{Z}_m$ by 
\begin{equation}\label{antimorphism}
\Psi_x(k) = x-k\quad  \hbox{for every } \ k\in \mathbb{Z}_m\,.
\end{equation}
This group  - usually called the {\it dihedral group of order $2m$}  -
contains $m$ morphisms and $m$ antimorphisms. As shown in
\cite{Sta2011},  if  $w$ is a factor of ${\bf t}_{b,m}$, then
$\nu(w)$ is a factor of ${\bf t}_{b,m}$ for every element $\nu$ of
the group  $I_2(m)$.
\vskip 0.3cm

Let us list  some properties of the generalized Thue--Morse word we will use later. They are not hard to  observe.
 (See also \cite {Sta2011}.) \\[1mm]

\noindent {\bf Properties } of ${\bf t}_{b,m}$

\begin{enumerate} 
\item \label{prop1} Let  $b \neq_m 1$. If $v = v_0v_1 \cdots v_{k-1}$  is a factor of  ${\bf
t}_{b,m}$ of length $k \geq 2b+1$, then there exists $j\in \{0,1,
\ldots,k-2\}$ such that $v_{j}+1 \neq_m v_{j+1}$.  Such index $j$
will be called {\it jump} in $v$. It is important to note here that we always start
indices from $0$. Sometimes, when no confusion can occur,
we will say that there is a jump between the letters $v_j$ and
$v_{j+1}$.

\item If $b \neq_m 1$, then a factor $v $ of length at least
$2b+1$  has uniquely determined $\varphi$-ancestors.

\item If $v = v_0v_1 \ldots v_k$   is a $\Psi$-palindrome and an
index $j$ is a jump in $v$,  then also the index $k-j$ is a jump in
$v$.

\item\label{komutace1} $\Psi_x\varphi =  \varphi \Psi_{x-b+1}$ for every $x \in \mathbb{Z}_m$.

\item\label{komutace2} For every $\Psi  \in I_2(m)$ there exists a unique $\Psi' \in I_2(m)$ such that $\Psi \varphi = \varphi
\Psi'$.

\item If $w\neq \varepsilon$ is a $\Psi$-palindrome for some $\Psi
\in I_2(m)$, then for every antimorphism $\Psi' \in I_2(m)$ such that $\Psi'
\neq \Psi$ we have $\Psi'(w)\neq w$.

\item If $\Psi\neq \Psi'$, then for every letter $a \in \mathcal{A
}$,  $\Psi(a)\neq  \Psi'(a)$.
\end{enumerate}

\section{Proof of Theorem  \ref{main}}

Proof of Theorem \ref{main} will be split into  Propositions
\ref{JdeTo} and \ref{NejdeTo}.

\begin{prop}\label{JdeTo} Let $m,b \in \mathbb{Z}$.
Denote \begin{equation}\label{directive}\Delta =
0\Bigl(12\ldots(b-1)\Bigr)^\omega \in \mathbb{Z}_m^\mathbb{N}
\quad \hbox{and} \quad \Theta= \Bigl(\Psi_0\Psi_1 \ldots
\Psi_{m-1}\Bigr)^\omega\,
 \in I_2(m)^\mathbb{N}\,.
 \end{equation}
If $b\leq m$ or $b =1 \pmod m$, then the generalized
pseudostandard word ${\bf u}_\Theta(\Delta) $  with directive
sequences $\Delta$ and $\Theta$ equals ${\bf t}_{b,m }$.

\end{prop}

Axel Thue found the classical Thue--Morse word  ${\bf t}_{2,2}$
when he searched for infinite words without overlapping factors, i.e.,
words without factors of the form $v = ws = pw$ such that $|w| > |s|$.
The authors of \cite{MaBrGlLa}  showed that  the generalized
Thue--Morse word  ${\bf t}_{b,m }$ is overlap-free if and only
if $b\leq m$. It is worth to mention that the same condition
appears in our characterization of non-periodic words ${\bf t}_{b,m }$ which are the
generalized pseudostandard words.

For parameters $b=m=2$,  Proposition \ref{JdeTo}  was shown in
\cite{LuLu}.   The following example illustrates that the
assumption $b\leq m$ is crucial for validity of Proposition
\ref{JdeTo}.

\begin{example}\label{example3}
    Consider  $b=4$ and $m=2$. On the alphabet $\mathcal{A} = \mathbb{Z}_2$, we have $\Psi_0(k) =
0-k = k$ and $\Psi_1(k) = 1-k$ for any letter $k$. In the notation
of Example \ref{example1},  it means $\Psi_0 = R$ and $\Psi_1 =
E$. Therefore the   sequences  $\Delta$ and  $\Theta$ from
Proposition \ref{JdeTo}  coincide with sequences $\Delta$ and
$\Theta$ from Example \ref{example2}. The generalized Thue--Morse  word ${\bf
t}_{4,2}$ starts as $${\bf t}_{4,2} =
01011010010110101010010110100101010110100101 \ldots $$ Note that,
using the notation from Definition \ref{def:GPS}, $w_8$ is not a
prefix of ${\bf t}_{4,2}$ and thus the generalized pseudostandard word
${\bf u}_\Theta(\Delta)$ from Proposition \ref{JdeTo} does not correspond to ${\bf
t}_{4,2}$.
\end{example}

\medskip

Propositions \ref{JdeTo} and \ref{NejdeTo} rely on several technical lemmas. The first one
settles the case for  periodic Thue--Morse words.

\begin{lemma}
Let $b =_m 1$. The word ${\bf u}_\Theta(\Delta)$ with the
directive sequences   $\Delta$ and  $\Theta$  given in
\eqref{directive} equals ${\bf t}_{b,m}$.
\end{lemma}

\begin{proof}

Let $n=\sum_{i=0}^ka_ib^i$  be the expansion of the number $n$ in
the base $b$. The assumption $b =_m 1$ implies  $b^i =_m 1$ for
any $i \in \mathbb{N}$. With respect to \eqref{DefThueMorse}, we
can write

$$ {\bf
t}_{b,m}(n) =_m s_b(n) = \sum_{i=0}^ka_i  =_m \sum_{i=0}^ka_ib^i =
n\,.
$$
Since  $\Delta = 0\Bigl(1\cdots(b-1)\Bigr)^\omega$ equals $(01
\cdots (m-1))^\omega$, we have showed that  ${\bf t}_{b,m} =
\Delta$.  Moreover, the sequence of antimorphisms  $\Theta$ can be
indexed by natural numbers as $\Theta = \Psi_0\Psi_1\Psi_2\Psi_3
\cdots $ where $\Psi_n = \Psi_x$ for $n=_m x$. Clearly by
\eqref{antimorphism}
 \begin{equation}\label{periodic}
\Psi_n(012\ldots n) = 012 \ldots n = (012 \ldots n)^{\Psi_n}\,.
 \end{equation}

Let the words $w_n$ have the meaning as in Definition
\ref{def:GPS}. We will show  by induction that  $w_{n+1} =
0123\cdots n$ for any $n \in \mathbb{N}$. We have $w_1 =
(0)^{\Psi_0} = 0 = \Psi_0(0)$. Using  definition of $w_{n+1}$ and
\eqref{periodic} we get
$$w_{n+1} = \bigl(w_n\delta_{n+1}\bigr)^{\Psi_{n}} = \Bigl(\bigl(012\cdots
(n-1)\bigr) n\Bigr)^{\Psi_{n}}  = 012\cdots (n-1)n\,.$$ This means
that ${\bf t}_{b,m} = \lim\limits_{n\to \infty} w_n$, as desired.
\end{proof}

We can now concentrate on the non-periodic Thue--Morse words, i.e., on the case $b \ne_m 1$, which will be treated using several lemmas.

\begin{lemma}\label{pomocne}    If  $\Psi
\in I_2(m)$ is an antimorphism   and  $p\in  \mathcal{A}^*$ is a
$\Psi$-palindrome such that $\varphi(a_1a_2)a_3$ is a suffix of
$p$ for some letters $a_1,a_2,a_3 \in  \mathcal{A}$,  then there
exists a word $w$ of length at least $2$ and antimorphism $\Psi'
\in I_2(m)$ such that
$$p = \Psi(a_3)\varphi(w)a_3\,, \quad \Psi'(w)=w \quad \hbox{and}
\quad \Psi \varphi = \varphi \Psi'.$$
\end{lemma}

\begin{proof}  Let $p=p_0p_1 \ldots p_{n}$.
Since $p$ has a suffix $\varphi(a_1a_2)a_3$ of length $2b+1$, according to Property \ref{prop1}, $p$ has a jump position.
The jump position of $p$ is either $n-1$ or $n-b-1$.
As $p$ is  a $\Psi$-palindrome, the index $0$ or $b$ is a jump of $p$.
It implies that a prefix of $p$ is of the form $a_3'\varphi(a_2'a_1')$ for some letters
$a_1',a_2',a_3'$. Thus $p =  \Psi(a_3)\varphi(w)a_3$ for some word
$w$ with $|w|\geq 2$.
As $p = \Psi(p)$, using Property \ref{komutace2}, we get $\varphi(w) =
\Psi(\varphi(w)) = \varphi \Psi'(w)$. Since $\varphi$ is injective,
we have $w =\Psi'(w)$.
\end{proof}

\begin{lemma}\label{lemmaVic}  Fix  $n\in \mathbb{N}$ and   $k \in \{2,\ldots, b-1\}$.  Put $\Psi = \Psi_{(b-1)n+k}$.  The longest $\Psi$-palindromic
 suffix  of the factor
$$
v = \varphi^n(0)\varphi^n(1) \varphi^n(2)\ldots \varphi^n(k-1)\,k
$$ is \  $\Psi(k)\varphi^n(1)\varphi^n(2) \ldots \varphi^n(k-1)\,k
\,$   and thus
\begin{equation}\label{suffixKN}
v^{\Psi} = \varphi^n(0)\varphi^n(1) \varphi^n(2)\ldots
\varphi^n(k-1)\varphi^n(k)\,.
\end{equation}
\end{lemma}
\begin{proof}  First we show that $u=\Psi(k)\varphi^n(1)\varphi^n(2) \ldots \varphi^n(k-1)\,k
\,$ is a  $\Psi$-palindromic
 suffix  of the factor $v$. It is easy to check that the last letter of $\varphi^n(0)$
is the letter $(b-1)n$. In our notation  $\Psi (k) =  {(b-1)n+k} - k
=(b-1)n$. Therefore $\Psi (k)$ is the last letter of
$\varphi^n(0)$, and thus $u$ is a suffix of $v$. To show that $u$
is a $\Psi$-palindromic suffix, we need to show
$$ \Psi\Bigl(\varphi^n(i)\Bigr) = \varphi^n(k-i) \quad \hbox{for any
}\ i=1,2\ldots, k-1\,.$$ Using Property \ref{komutace1},  we get
$$\Psi\varphi^n = \Psi_{(b-1)n+k}  \varphi^n = \varphi^n \Psi_k,
$$
and thus $ \Psi\Bigl(\varphi^n(i)\Bigr) =\varphi^n \Bigl(\Psi_k
(i)\Bigr) = \varphi^n (k-i)$ for all $i$, including $i=0$ and $i=k$. \\

 Now we show by contradiction that $u$ is
the longest $\Psi$-palindromic
 suffix  of $v$. Consider the minimal $n$ for which  the statement
 is false, i.e.,  the longest  $\Psi$-palindromic
 suffix  of $v$ - denote it by $p$ - is longer than   $u$.
Since  $ 01 \ldots (k-1) k = (01 \ldots (k-1) k) ^\Psi$, the
minimal $n$ is $\geq 1$. As  $|p| > |u| $,  $p$ has  a suffix $u$
and  we can apply   Lemma \ref{pomocne}.  Therefore $p$ has the
form $p = \Psi(k) \varphi(w) k$, where $w$ is a $\Psi'$-palindrome
and $\varphi \Psi' = \Psi \varphi$. According to Property 4, we
have $\Psi' = \Psi_{(b-1)(n-1)+k}$. In particular, $\Psi'(k) w k$
is $\Psi'$-palindromic suffix of $\varphi^{n-1}(0)
\varphi^{n-1}(1)\varphi^{n-1}(2) \ldots \varphi^{n-1}(k-1)k$.

As $|p| = 2+ |\varphi(w)|> |u| $, necessarily $ |w| >
|\varphi^{n-1}(1)\varphi^{n-1}(2) \ldots \varphi^{n-1}(k-1)|$. It
means that $\varphi^{n-1}(0) \varphi^{n-1}(1)\varphi^{n-1}(2)
\ldots \varphi^{n-1}(k-1)k$ has the longest $\Psi'$-palindromic
suffix longer than $\Psi'(k) \varphi^{n-1}(1)\varphi^{n-1}(2)
\ldots \varphi^{n-1}(k-1)k$ \ -  \ contradiction with the minimality
of $n$.
\end{proof}

\begin{lemma}\label{konec1} 
 Let $v$ be a factor with the suffix $\varphi\bigl((a-1)a\bigr) 1$ and let $b \neq_m 1$. 
Put $\Psi = \Psi_{a+b}$. 
Under these assumptions, the longest $\Psi$-palindromic suffix $p$ of the factor $v$ is of length at least 2.
Moreover, for  $|p|$ and the parameters $a$ and $b$, the following holds:
\begin{enumerate}
\item if $|p| \geq b + 1$, then  $p= \Psi(1)\varphi(w)1$, where $w$ is a $\Psi_{a+1}$-palindrome of length at least $2$;
\item if $3 \leq |p| \leq b$, then $a+b=_m 1$ and $b> m$;
\item if $|p|=2$, then  either  $a+b \neq_m 1$ or  $a+b =_m 1$  with $b\leq m$.

\end{enumerate}
\end{lemma}

\begin{proof}

Since the last two letters of $v$ are $(a+b - 1)1$, and $\Psi (1) = a + b - 1$, the word $v$ has a palindromic suffix of length $2$.

If $p$ itself has the suffix $\varphi\bigl((a-1)a\bigr) 1$, then the form of $p$ is given by   Lemma \ref{pomocne} as $p = \Psi(1) \varphi(w)1$.  According to
Properties \ref{komutace1} and \ref{komutace2} in Section \ref{T-M}, we have
$\Psi_{a+b}\varphi = \varphi\Psi_{a+1}$ and thus $w$ is a
$\Psi_{a+1}$-palindrome.\\

Let $p$ be shorter than the suffix $\varphi\bigl((a-1)a\bigr) 1$.
It means that $p$ is a suffix of the factor $a(a+1)\ldots
(a+b-2)a(a+1) \ldots (a+b-1)1.$   Since $b-1\neq_m 0 $, we have a
jump between letters $a+b-2$ and $a$. Let us discuss the following two cases
separately:

i)  If  $a+b\neq_m 1$, then the other jump is between the last two
letters $a+b-1$ and  $1$. In the $\Psi$-palindrome, 
jump positions must be symmetric with respect to the center, and thus
the only two  candidates for the palindromic suffix are  $(a+b-2)
a(a+1)\ldots (a+b-1)1$ and $(a+b-1)1$. Since $\Psi(1) =
\Psi_{a+b}(1) \ne_m a+b-2$, only the latter possibility  $p= \Psi(1)1$
occurs.

ii) If  $a+b=_m  1$, then $a(a+1)\ldots (a+b-2)a(a+1) \ldots
(a+b-1)1$  has only one jump, namely, as we mentioned above, between letters $a+b-2$ and $a$.  Therefore
the longest palindromic suffix $p$ does not contain any jumps. It
implies, that $p$ is a suffix of $a(a+1)\ldots (a+b-1)1$.  Let $k
\in \{0,1,\ldots, m-1\}$ be a letter such that $p=(a+k)(a+k+1)
\ldots (a+b-1)1$. Then $a+k = \Psi(1) = a+b-1 =_m 0$. Or
equivalently,  $k=_m b-1$.

If $b\leq m$,  the equality $k=_m b-1$ has the only solution
$k=b-1$, i.e., $p= (a+b-1)1 = \Psi(1)1$, as before.

If $b> m$, then  the smallest  $k \in\{0,1,\ldots, b-1\}$  solving
$k =_m b-1 $, satisfies $k \leq b-1-m \leq  b-3$, and as well $k > 0$
(since $ k = b-1 \neq_m 0$). As  $p=(a+k)(a+k+1) \ldots
(a+b-1)1$ has the length  $|p| = b-k+1$,  we get $3\leq  |p| \leq b$.
\end{proof}

The following claim addresses the question of the length of the longest $\Psi$-palindromic suffix
of the factor $\varphi^n(0)1$. 
\begin{claim}\label{claim2} Let $b \neq_m 1$. Put $q = \min\{ i \in \mathbb{N} \colon i> 0 \text{ and } i (b-1) =_m 0 \}$ and for a fixed  $n\in \mathbb{N}$ denote  $\Psi =
\Psi_{(b-1)n+1}$.

\begin{enumerate}
\item If $b\leq m$, then the longest $\Psi$-palindromic suffix of the factor $\varphi^n(0)1$ is of length $2$.

 \item If $b>m$ and $n< q$, then the longest $\Psi$-palindromic suffix of the factor $\varphi^n(0)1$ is of length $2$.

\item If $b>m$   and $n=q$,  then the longest $\Psi$-palindromic suffix of the factor $\varphi^n(0)1$ is of length greater than $2$ and less than $b+1$.
\end{enumerate}

\end{claim}

\begin{proof}   As $b-1 \neq_m 0$, we have $q\geq 2$.   First we show that Claim holds for $n=0$ and
$n=1$.

Consider $n=0$.
The factor $01$ is a $\Psi_1$-palindrome as $\Psi_1(0) = 1$.

Consider $n=1$. The factor $\varphi(0)1 = 01\ldots (b-1)1$ has
only one  jump, namely  between two last letters $b-1$ and $1$
because of  $b-1 \neq_m 0$.  Thus $\varphi(0)1$ cannot have
$\Psi$-palindromic suffix longer than $2$.  Since $\Psi(b-1) =
\Psi_{b} (b-1) = 1$, the factor $\varphi(0)1$ has the longest
$\Psi$-palindromic suffix of length  $2$.

Now suppose that there exists an index $n$ such that
$\varphi^n(0)1$  has a $\Psi$-palindromic suffix of length at
least  $3$. Consider the smallest such $n$. Obviously, $n \geq 2$.
Denote by $p$  the longest  $\Psi$-palindromic suffix  of
$\varphi^n(0)1$. We will apply  Lemma \ref{konec1} with $v =
\varphi^n(0)1 = \varphi\Bigl(\varphi^{n-1}(0)\Bigr)1 $. Since the
last letter of $\varphi^{n-1}(0)$ equals to $(n-1)(b-1)$, we
denote  $a= (n-1)(b-1)$. For this choice of $a$,   the
antimorphism  $\Psi = \Psi_{(b-1)n+1}=  \Psi_{a+b}$ is as required in
 Lemma \ref{konec1}.  Since $|p|\geq 3$, only Cases 1 and  2  from Lemma \ref{konec1}
 apply:\\

Case 1: \ \  $p= \Psi(1)\varphi(w)1 = (a+b-1)\varphi(w)1$ for some
factor $w$ of length  $|w|\geq 2$ and $w$  a
$\Psi_{a+1}$-palindrome.

Let us realize that  $\varphi^n(0)1$  is a prefix of  ${\bf
t}_{b,m}$ for any $n$ and thus $\varphi^n(0)\varphi(1)$ is its
prefix as well. Since $(a+b-1)$ is the last letter of $\varphi(a)$,
we can deduce that  $\varphi(a)\varphi(w)\varphi(1)$ is a suffix
of $\varphi^n(0)\varphi(1)$ and thus $aw1$ is a suffix of
$\varphi^{n-1}(0)1$. Moreover, $aw1$ is a $\Psi_{a+1}$-palindrome
of length $\geq 4$. This means that $\varphi^{n-1}(0)1$ has a
$\Psi'$-palindromic suffix of length greater than 2, where $\Psi'=
\Psi_{a+1} = \Psi_{(n-1)(b-1) +1}$. This is a contradiction with the minimality of $n$.\\

Case 2: \ \ $a+b=_m 1$, $b> m$.

  Since we denoted $a=
(n-1)(b-1)$, we have $n(b-1) =_m 0$.   The smallest $n$
satisfying this equality was denoted by $q$. \\

We can conclude:  If $b\leq m$ then for all $n$, the longest
$\Psi$-palindromic suffix  of $\varphi^n(0)1$has length $2$; 
if  $b>m$ then for  all
$n <q$, the longest $\Psi$-palindromic suffix  $\varphi^n(0)1$ has length $2$; 
if $b>m$  and $n=q$, then the longest $\Psi$-palindromic suffix  $\varphi^n(0)1$ has
length  among numbers $3,4, \ldots , b$.
\end{proof}

\begin{lemma}\label{konecN1} Let $b \neq_m 1$. Denote $q = \min\{ i \in \mathbb{N} \colon i> 0 \text{ and } i (b-1) =_m 0 \}$. Fix   $n\in \mathbb{N}$ and   put $\Psi =
\Psi_{(b-1)n+1}$.

\begin{enumerate}

\item If $b\leq m$,  then $\Bigl(\varphi^n(0)1\Bigr)^\Psi =
\varphi^n(0)\varphi^n(1)$.

\item If $b> m$ and $n < q$, then $\Bigl(\varphi^n(0)1\Bigr)^\Psi =
\varphi^n(0)\varphi^n(1)$.

\item If $b> m$,  then $
 \Bigl(\varphi^q(0)1\Bigr)^\Psi$ is not a prefix of  ${\bf t}_{b,m}.$
 \end{enumerate}
\end{lemma}
\begin{proof}  Let us denote by $s$ the length of the longest $\Psi$-palindromic suffix of
$\varphi^n(0)1$. We will apply Claim \ref{claim2}.

If $s=2$, we clearly have   $\Bigl(\varphi^n(0)1\Bigr)^\Psi  =
\varphi^n(0) \Psi \bigl( \varphi^n(0) \bigr)$.  According to
Property \ref{komutace1}, $\Psi \varphi^n = \varphi^n \Psi_{1}$
and thus $\Psi \bigl( \varphi^n(0)\bigr) =
\varphi^n\bigl(\Psi_1(0)\bigr) =\varphi^n(1)$.

Consider now  $s \in\{3,4,\ldots,b\}$.  Then
$\Bigl(\varphi^n(0)1\Bigr)^\Psi $ has length
$2|\varphi^n(0)|+2-s$. 
From  the form of the substitution $\varphi$, and the fact that  ${\bf t}_{b,m}$ is its fixed point, it follows that 
a  jump   in  ${\bf t}_{b,m} = u_0u_1u_2\cdots$
can occur only   on indices $i-1 =_b -1 $.
 Since $\varphi^n(0)$ is a prefix of  ${\bf t}_{b,m}$,
 the prefix  of the  palindrome $\Bigl(\varphi^n(0)1\Bigr)^\Psi$
 with length $| \varphi^n(0)|$, has jumps on positions $i-1 =_b -1$. Jumps in any palindrome occur symmetrically with respect to
the center of the palindrome.  The length of the palindrome
$\Bigl(\varphi^n(0)1\Bigr)^\Psi$  is $(2-s) \bmod b$.  As $2-s\neq_b
0$, jumps in the left part of
$\Bigl(\varphi^n(0)1\Bigr)^\Psi$ are not compatible  with  the jump
positions  in  ${\bf t}_{b,m}\,$ and thus
$\Bigl(\varphi^n(0)1\Bigr)^\Psi$  cannot be a prefix of ${\bf
t}_{b,m}.$
\end{proof}

Now we are ready to complete the proof of Proposition \ref{JdeTo}  for the
non-periodic Thue--Morse words.

\begin{proof}
From  Lemma \ref{konecN1}, Part 1, we get the first identity in the following list; the others follow from Lemma \ref{lemmaVic}:
\begin{description}
\item[1)] $ \Bigl(\varphi^n(0)1\Bigr)^\Psi =
\varphi^n(0)\varphi^n(1)$ \quad if  $\Psi = \Psi_{(b-1)n+1}$ \ \
{\rm and} \ \  $b\leq m$;

\item[2)] $ \Bigl(\varphi^n(0)\varphi^n(1)2\Bigr)^\Psi =
\varphi^n(0)\varphi^n(1)\varphi^n(2)$  \quad  if $\Psi =
\Psi_{(b-1)n+2}$;

\item[$\vdots$]

\item [b-1)]$ \Bigl(\varphi^n(0)\varphi^n(1) \ldots \varphi^n(b-2)
(b-1)\Bigr)^\Psi = \varphi^n(0)\varphi^n(1) \ldots \varphi^n(b-1)$
\quad if  $\Psi = \Psi_{(b-1)n+b-1}$.
\end{description}

\noindent Since ${\bf t}_{b,m} = \lim\limits_{n \to \infty} \varphi^n(0)$, this together with the simple fact
$$ \varphi^n(0)\varphi^n(1) \ldots \varphi^n(b-1)= \varphi^{n}\bigl(\varphi(0)\bigr) = \varphi^{n+1}(0) $$ finishes the proof of Proposition \ref{JdeTo}.
\end{proof}

\begin{prop}\label{NejdeTo} Let $m,b \in \mathbb{Z}$.
If $b> m$ and  $b \neq 1 \pmod m$, then  ${\bf t}_{b,m }$
 is not a generalized pseudostandard word.
\end{prop}

\begin{proof}
First, we show that a pseudopalindromic prefix of ${\bf t}_{b,m }$ which is longer than $b$ is an image of a shorter pseudopalindromic prefix of ${\bf t}_{b,m }$.

Since the word ${\bf t}_{b,m } = 01 \cdots (b-1)1 \ldots $ has its first jump equal to $b-1$, every its pseudopalindromic prefix $p$ longer than $b$ has a jump $|p|-b$.
This implies that $p = \varphi(p')$ for some prefix $p'$.
Since $\Psi(p) = p$ for some antimorphism $\Psi \in I_2(m)$, according to Property 5, we have $\Psi(\varphi(p')) = \varphi(\Psi'(p')) = p = \varphi(p')$  for some antimorphism $\Psi' \in I_2(m)$.
Since $\varphi$ is injective, the last equality implies $\Psi'(p') = p'$, and thus $p'$ is a $\Psi'$-palindromic prefix.

One can see that  for all $n$, $\varphi^n(0)$ and $\varphi^n(01)$ are pseudopalindromic prefixes.
Next, we show  that for each $n \in \N$, the only palindromic prefix of ${\bf t}_{b,m }$ which is longer than
$|\varphi^n(0)|$ and shorter than $2|\varphi^n(0)| + 2$, is the prefix $\varphi^n(0)\varphi^n(1)$.

This part of the proof will proceed by contradiction: Suppose that $n$ is the minimal integer for which the claim does not hold.
Clearly $n > 1$, since the claim can  be easily verified for $n=1$.
Using the fact that every pseudopalindromic prefix of ${\bf t}_{b,m }$ is a $\varphi$-image of a shorter one, we can immediately see that even for $n-1$ the statement does not hold, which is a contradiction with the minimality of $n$.

Since $b > 2$, there is no pseudopalidromic prefix of length $|\varphi^n (0)|-1$.
For the lengths of the words $w_i$ from Definition \ref{def:GPS}, we have that $|w_{i+1}| \leq 2|w_i|+2$ for all $i$.
Therefore, for each $n$, there exists an index $i$ such that $w_i = \varphi^n(0)$ and $w_{i+1} =\varphi^n(0)\varphi^n(1)$.
Let $\Psi$ be the antimorphism which fixes $w_{i+1}$, i.e., $w_{i+1} = (w_i1)^\Psi$.
The lengths of $w_i$ and $w_{i+1}$ imply that the longest $\Psi$-palindromic suffix of $w_i1$ is of length $2$.

Since the last letter of $\varphi^n(0)$ is the letter $n(b-1)$, the antimorphism $\Psi$ satisfies $\Psi (1) = n(b-1)$ and thus $\Psi = \Psi_{n(b-1) + 1}$.

Set $n = q$ where $q$ is the order of $(b-1)$.
It follows from Part 3 of Lemma \ref{konecN1}, that the $\Psi_{q(b-1) + 1}$-palindromic closure of $w_i1$ is not a prefix of ${\bf t}_{b,m}$.
  \end{proof}

\section{Comments and open questions}

\begin{enumerate}

\item As shown in Proposition  \ref{JdeTo}, the word ${\bf
t}_{3,4}$ is a generalized pseudopalindromic word and its directive
sequences are $\Delta = 0(12)^\omega$ and $\Theta=
\Bigl(\Psi_0\Psi_1 \Psi_2\Psi_3\Bigr)^\omega$\,.  One can easily
check that the pairs

$$\Delta = 0(21)^\omega, \ \ \Theta= \Bigl(\Psi_1
\Psi_2\Psi_3\Psi_0\Bigr)^\omega \quad {\rm and} \quad \Delta =
01(12)^\omega, \ \ \Theta= \Psi_0 \Psi_2\Psi_3\Bigl( \Psi_0\Psi_1
\Psi_2\Psi_3\Bigr)^\omega$$

also correspond to the word ${\bf t}_{3,4}$.

The authors of  \cite{BlPaTrVu} study this phenomenon for the
generalized pseudopalindromic word on the binary alphabet, where
$\Delta \in \{0,1\}^\mathbb{N}$ and  $\Theta \in
\{R,E\}^\mathbb{N}$. They defined the notion of a normalized bisequence
and showed (Theorem 27 in  \cite{BlPaTrVu}) that every pseudostandard
word is generated by a unique normalized bisequence. Moreover, for
any generalized pseudopalindromic word  ${\bf u}_{\Theta}(\Delta)$,
a simple algorithms which transforms the pair $\Delta$, $\Theta$ into
the normalized bisequence  is given.

{\bf Question:}\ \  Is it possible to generalize the notion of a
normalized bisequence for the case of a multi-literal alphabet?

\item It is well known the factor complexity of standard episturmian words is bounded by
$(\# \mathcal{A}-1)n + 1$. In
particular, on binary alphabet these words
which are not periodic are precisely standard Sturmian words and
their factor complexity is $\mathcal{C}(n)= n +1$.

In \cite{BlPaTrVu}, the authors conjectured that generalized
pseudostandard words on binary alphabet have their factor complexity bounded by $4n + \, const$.

The factor complexity of binary generalized Thue--Morse words can
be found in \cite{TrSh}. The word ${\bf t}_{2k+1,2}$ is periodic,
and thus its factor complexity is bounded by a constant.  The word
${\bf t}_{2k,2}$ is aperiodic and its factor complexity is $\leq
4n$ for any parameter $k$. It means that even ${\bf t}_{4,2}$ and
${\bf t}_{6,2}$ (which are not generalized pseudopalindromic
words) have a small complexity. It, of course, does not
contradict the conjecture.

The factor complexity of generalized Thue--Morse words  on any
alphabet is deduced in \cite{Sta2011}.  If the word ${\bf t}_{b,m}$
is aperiodic, then
$$ (qm-1)n\leq \mathcal{C}(n) \leq qmn\,,$$
where $q$ is the order of $b-1$ in the additive group
$\mathbb{Z}_m$, i.e. $q$ is the minimal positive integer such that
$q(b-1) =_m 0$.

The factor complexity of any infinite word  can be derived from
knowledge of its bispecial factors. Each aperiodic standard
episturmian word ${\bf u}$ has a nice structure of its bispecial
factors. (A factor $w$ is bispecial if and only if $w$ is a
palindromic prefix of ${\bf u}$.)

{\bf Question:}\ \  Is it possible to describe the structure of bispecial
factors for a generalized pseudostandard word?

\item It is known \cite{DrJuPi} that  classical standard palindromic
words with a periodic directive sequence $\Delta =
(\delta_1\delta_2\ldots \delta_k)^\omega$ are invariant under a
substitution.  For example, the Tribonacci word has the directive
sequence $\Delta = (012)^\omega$ and simultaneously, it is a fixed
point of the substitution $\varphi:  0\mapsto 01, 1 \mapsto 02,
2\mapsto 0$.

Let us denote ${\bf s}_{b,m}$ the generalized pseudostandard word with
$$\Delta = 0\Bigl(12\ldots(b-1)\Bigr)^\omega \quad \hbox{and}
\quad \Theta= \Bigl(\Psi_0\Psi_1 \ldots
\Psi_{m-1}\Bigr)^\omega.$$

If $b\leq m$, then ${\bf s}_{b,m} = {\bf t}_{b,m}$ and obviously
${\bf s}_{b,m} $ is invariant under the substitution described in
\eqref{substitution}.

{\bf Question:}\ \  Is the word  ${\bf s}_{b,m}$ a fixed point of a
substitution if $b>m$?

\end{enumerate}

\section*{Acknowledgements}
The second author acknowledges financial support from the Czech Science Foundation grant 13-03538S
and the last author acknowledges financial support from the Czech Science Foundation grant 13-35273P.

\bibliographystyle{amsplain}
\IfFileExists{biblio.bib}{\bibliography{biblio}}{\bibliography{../!bibliography/biblio}}

\end{document}